\documentclass[11pt]{article}
\usepackage{fontspec}
\usepackage[margin=1in]{geometry}
\usepackage{amsmath,amssymb,amsthm,array,booktabs}
\usepackage[explicit]{titlesec}

\newtheorem{thm}{Theorem}
\newtheorem{lem}[thm]{Lemma}
\newtheorem{cor}[thm]{Corollary}

\titleformat{\subsection}[runin]
  {\normalfont\large\bfseries}
  {\thesubsection}{1em}{#1.}
\titlespacing*{\subsection}{0pt}
  {3.25ex plus 1ex minus .2ex}{.6em}

\title{Non-derivability of Euclidean Division in
$\mathrm{PA}_{\mathrm{smu}}^{-}$}
\author{Naoya~Kato}
\date{}

\begin{document}

\maketitle

\begin{abstract}
$\mathrm{PA}_{\mathrm{smu}}^{-}$ is a weak theory of arithmetic obtained by
adding two principles concerning powers of two to basic axioms satisfied by
the nonnegative part of a discretely ordered ring. We introduce a theory
$\mathrm{PA}_{\mathrm{wit}}^{-}$ in which powers of two and the required
witnesses are represented by primitive symbols. We show that, along a fixed
sequence in the standard model of arithmetic, every one-variable term
eventually agrees with a polynomial over the dyadic rationals. A finite
avoidance lemma and the Compactness Theorem then yield a model of
$\mathrm{PA}_{\mathrm{smu}}^{-}$ in which division by $3$ fails.
In this model even the $n=3$ instance of the Standard Euclidean Division
Principle pa16 fails. In fact, the model can be chosen to satisfy every
universal $\mathcal L_0$-sentence true in the standard model. Consequently,
neither pa16 nor the Euclidean Division Principle pa17 is derivable even
after these universal truths are adjoined to
$\mathrm{PA}_{\mathrm{smu}}^{-}$.
\end{abstract}

\section{Introduction}

In Open Question 3.2, Visser asked whether
$\mathrm{PA}_{\mathrm{smu}}^{-}$ proves the Euclidean Division Principle
pa17~\cite{visser}. The purpose of this paper is to show that
\[
\mathrm{pa17}:\quad
\forall x\,\forall y\bigl(y\ne0\rightarrow
\exists q\,\exists r\,(r<y\land x=q\cdot y+r)\bigr)
\]
is not derivable from $\mathrm{PA}_{\mathrm{smu}}^{-}$. In fact, we
construct a model in which division by $3$ fails: it contains an element
that is not of the form $3a+j$ for any $a$ and any $j\in\{0,1,2\}$.
This also refutes the $n=3$ instance of the Standard Euclidean Division
Principle pa16. Thus the paper not only answers Visser's question
negatively, but also shows that $\mathrm{PA}_{\mathrm{smu}}^{-}$ does not
prove pa16. The construction gives the stronger result that the
countermodel may be chosen to satisfy every universal $\mathcal L_0$-sentence
true in the standard model of arithmetic.

To work with generated substructures, we represent being a power of two by
the quantifier-free equation $\operatorname{obst}(x)=1$, together with
function symbols for the witnesses required by pa15, pa19, pa20, and for
the correctness of this equation. The resulting theory is denoted by
$\mathrm{PA}_{\mathrm{wit}}^{-}$. We then fix the standard model of
arithmetic and represent the values of terms along a sequence of powers of
two by polynomials. This yields simultaneous avoidance for any finite
collection of terms, after which the Compactness Theorem completes the
construction. For general background on arithmetic and its model theory, see
Hájek--Pudlák~\cite{hajek-pudlak} and Kaye~\cite{kaye}.

\section{The base theory}

In this section, we fix the language and notation, recall the axioms of
$\mathrm{PA}_{\mathrm{smu}}^{-}$, and state the division principles
considered below.

\subsection{Language and notation}
Let the base language of arithmetic be
\[
\mathcal L_0=\{0,1,+,\cdot,\le\}.
\]
We use the following abbreviations throughout:
\begin{itemize}
\item $2:=1+1$ and $3:=2+1$;
\item $x<y\;:\Leftrightarrow\;x\le y\land x\ne y$;
\item $x\mid y\;:\Leftrightarrow\;\exists z\,(z\cdot x=y)$;
\item $\operatorname{pow}_2(x)\;:\Leftrightarrow\;
\forall y\bigl(y\mid x\rightarrow(y=1\lor 2\mid y)\bigr)$.
\end{itemize}

\subsection{The axioms of $\mathrm{PA}_{\mathrm{smu}}^{-}$}
$\mathrm{PA}_{\mathrm{smu}}^{-}$ is the $\mathcal L_0$-theory consisting of
pa1--pa13, pa15, pa19, and pa20 below. All free variables are understood to
be universally quantified.

\begin{align*}
\mathrm{pa1}:\quad &x+0=x,\\
\mathrm{pa2}:\quad &x+y=y+x,\\
\mathrm{pa3}:\quad &(x+y)+z=x+(y+z),\\
\mathrm{pa4}:\quad &x\cdot1=x,\\
\mathrm{pa5}:\quad &x\cdot y=y\cdot x,\\
\mathrm{pa6}:\quad &(x\cdot y)\cdot z=x\cdot(y\cdot z),\\
\mathrm{pa7}:\quad &x\cdot(y+z)=x\cdot y+x\cdot z,\\
\mathrm{pa8}:\quad &x\le y\lor y\le x,\\
\mathrm{pa9}:\quad &(x\le y\land y\le z)\rightarrow x\le z,\\
\mathrm{pa10}:\quad &x+1\not\le x,\\
\mathrm{pa11}:\quad &y\le x\rightarrow(y=x\lor y+1\le x),\\
\mathrm{pa12}:\quad &y\le x\rightarrow y+z\le x+z,\\
\mathrm{pa13}:\quad &y\le x\rightarrow y\cdot z\le x\cdot z,\\
\mathrm{pa15}:\quad &y\le x\rightarrow\exists z\,(x=y+z),\\
\mathrm{pa19}:\quad &\exists u\,
 \bigl(\operatorname{pow}_2(u)\land u\le x+1<2\cdot u\bigr),\\
\mathrm{pa20}:\quad &
 \bigl(\operatorname{pow}_2(x)\land
       \operatorname{pow}_2(y)\land x\le y\bigr)
 \rightarrow x\mid y.
\end{align*}

Write $\overline n$ for the $\mathcal L_0$-numeral denoting the standard
natural number $n$. The Standard Euclidean Division Principle pa16 is the
scheme consisting, for every $n\in\mathbb N\setminus\{0\}$, of the sentence
\[
\forall x\,\exists q\,\exists r\,
\bigl(r<\overline n\land x=q\cdot\overline n+r\bigr).
\]
We shall use, in particular, its $n=3$ instance
\begin{equation}
\tag{$\mathrm{pa16}_3$}
\forall x\,\exists q\,\exists r\,(r<3\land x=q\cdot3+r).
\end{equation}
The Euclidean Division Principle pa17 considered here is the sentence
\begin{equation}
\tag{pa17}
\forall x\,\forall y\bigl(y\ne0\rightarrow
\exists q\,\exists r\,(r<y\land x=q\cdot y+r)\bigr).
\end{equation}

We shall use the elementary consequences of pa1--pa13 that
\[
\forall x\,0\le x,
\qquad
3\ne0,
\qquad
r<3\rightarrow(r=0\lor r=1\lor r=2).
\]
Consequently, the
condition $r<y$ in pa17 is equivalent over the base theory to
$0\le r<y$. Thus our formulation is equivalent to the version in which
nonnegativity of the remainder is stated explicitly.

\section{The universal witness expansion $\mathrm{PA}_{\mathrm{wit}}^{-}$}

In this section, we introduce a universal expansion by witness functions
and establish its precise relation to $\mathrm{PA}_{\mathrm{smu}}^{-}$.

\subsection{Language}
The language of $\mathrm{PA}_{\mathrm{wit}}^{-}$ is
\[
\mathcal L_{\mathrm{wit}}=
\mathcal L_0\cup
\{\operatorname{sub},
  \operatorname{po},
  \operatorname{quo},
  \operatorname{half},
  \operatorname{obst},
  \operatorname{cof}\}.
\]
The arities and intended roles of the new symbols are as follows.

\[
\begin{array}{ccl}
\toprule
\text{symbol}&\text{arity}&\text{intended role}\\
\midrule
\operatorname{sub}&2&\text{cut-off subtraction}\\
\operatorname{po}&1&\text{a power-of-two witness for pa19}\\
\operatorname{quo}&2&\text{the quotient required by pa20, or $0$ otherwise}\\
\operatorname{half}&1&\text{the half of an even element, or $0$ otherwise}\\
\operatorname{obst}&1&\text{a divisor not divisible by $2$ (equal to $1$ for powers of two)}\\
\operatorname{cof}&1&\text{the cofactor corresponding to $\operatorname{obst}(x)$}\\
\bottomrule
\end{array}
\]

All the new symbols are total function symbols. Apart from the selection of
an obstruction divisor for a non-power of two, the axioms below determine
their intended values uniquely.

\subsection{Axioms}
$\mathrm{PA}_{\mathrm{wit}}^{-}$ is the $\mathcal L_{\mathrm{wit}}$-theory
consisting of the axioms listed below. All free variables are understood to
be universally quantified.

\subsubsection{Universal arithmetic axioms}

$\mathrm{PA}_{\mathrm{wit}}^{-}$ contains pa1--pa13.

\subsubsection{Cut-off subtraction}

\begin{equation}
\tag{$K_{\mathrm{sub}}$}
\begin{aligned}
y\le x&\rightarrow x=y+\operatorname{sub}(x,y),\\
x<y&\rightarrow\operatorname{sub}(x,y)=0.
\end{aligned}
\end{equation}

\subsubsection{The divisor condition in the power-of-two case}

\begin{equation}
\tag{$K_{\mathrm{half}}$}
\begin{aligned}
\operatorname{obst}(x)=1\land z\cdot y=x
&\rightarrow
\bigl(y=1\lor2\cdot\operatorname{half}(y)=y\bigr),\\
2\cdot w=y&\rightarrow\operatorname{half}(y)=w,\\
2\cdot\operatorname{half}(y)=y
&\lor\operatorname{half}(y)=0.
\end{aligned}
\end{equation}

If $\operatorname{obst}(x)=1$ and $y$ divides $x$, then either $y=1$ or
$\operatorname{half}(y)$ gives one half of $y$. The other two clauses make
$\operatorname{half}$ the unique half on even inputs and $0$ on odd inputs.

\subsubsection{The obstruction divisor and its cofactor}

\begin{align}
\tag{$K_{\mathrm{fac}}$}
&\operatorname{cof}(x)\cdot\operatorname{obst}(x)=x,\\
\tag{$K_{\mathrm{obst}}$}
&\operatorname{obst}(x)=1\lor
2\cdot w\ne\operatorname{obst}(x).
\end{align}

The axiom $K_{\mathrm{obst}}$ is universally quantified in $w$. Thus, if
$\operatorname{obst}(x)\ne1$, then $\operatorname{obst}(x)$ is a divisor of
$x$ different from $1$ and is not divisible by $2$.

\subsubsection{A witness for pa19}

\begin{equation}
\tag{$K_{\mathrm{po}}$}
\operatorname{obst}\bigl(\operatorname{po}(x)\bigr)=1
\land
\operatorname{po}(x)\le x+1
<2\cdot\operatorname{po}(x).
\end{equation}

\subsubsection{A witness for pa20}

\begin{equation}
\tag{$K_{\mathrm{quo}}$}
\begin{aligned}
\operatorname{obst}(s)=1\land \operatorname{obst}(r)=1\land s\le r
&\rightarrow\operatorname{quo}(s,r)\cdot s=r,\\
\neg\bigl(\operatorname{obst}(s)=1\land
\operatorname{obst}(r)=1\land s\le r\bigr)
&\rightarrow\operatorname{quo}(s,r)=0.
\end{aligned}
\end{equation}

We call the collection of pa1--pa13 and $K_{\mathrm{sub}}$,
$K_{\mathrm{half}}$, $K_{\mathrm{fac}}$, $K_{\mathrm{obst}}$, $K_{\mathrm{po}}$,
$K_{\mathrm{quo}}$ the theory $\mathrm{PA}_{\mathrm{wit}}^{-}$.

\subsection{The relation between $\operatorname{obst}(x)=1$ and $\operatorname{pow}_2$}
\begin{lem}\label{lem:power-test-correct}
In $\mathrm{PA}_{\mathrm{wit}}^{-}$,
\[
\operatorname{obst}(x)=1
\leftrightarrow\operatorname{pow}_2(x).
\]
\end{lem}

\begin{proof}
First suppose $\operatorname{obst}(x)=1$ and $y\mid x$. Then
$z\cdot y=x$ for some $z$. By $K_{\mathrm{half}}$, either $y=1$ or
\[
2\cdot\operatorname{half}(y)=y.
\]
Thus $y=1$ or $2\mid y$, proving $\operatorname{pow}_2(x)$.

Next suppose $\operatorname{obst}(x)\ne1$. By $K_{\mathrm{fac}}$,
\[
\operatorname{cof}(x)\cdot\operatorname{obst}(x)=x,
\]
so $\operatorname{obst}(x)$ is a divisor of $x$ different from $1$. If
$2\mid\operatorname{obst}(x)$, then for
some $w$ we have $w\cdot2=\operatorname{obst}(x)$, hence
$2\cdot w=\operatorname{obst}(x)$, contradicting $K_{\mathrm{obst}}$. Therefore
$2\nmid\operatorname{obst}(x)$, and $\neg\operatorname{pow}_2(x)$ follows.

Taking the contrapositive gives
$\operatorname{pow}_2(x)\rightarrow\operatorname{obst}(x)=1$.
\end{proof}

\subsection{Relation to $\mathrm{PA}_{\mathrm{smu}}^{-}$}
\begin{lem}\label{lem:reduct}
The following hold.
\begin{enumerate}
\item[(1)] The $\mathcal L_0$-reduct of every model of
$\mathrm{PA}_{\mathrm{wit}}^{-}$ is a model of
$\mathrm{PA}_{\mathrm{smu}}^{-}$.

\item[(2)] Every model of $\mathrm{PA}_{\mathrm{smu}}^{-}$ can be expanded to a
model of $\mathrm{PA}_{\mathrm{wit}}^{-}$.
\end{enumerate}
Consequently, $\mathrm{PA}_{\mathrm{wit}}^{-}$ is a model-preserving
expansion of $\mathrm{PA}_{\mathrm{smu}}^{-}$ by witness functions.
\end{lem}

\begin{proof}
Part (1) follows immediately from Lemma~\ref{lem:power-test-correct} and
the witness axioms: $K_{\mathrm{sub}}$, $K_{\mathrm{po}}$, and
$K_{\mathrm{quo}}$ provide witnesses for pa15, pa19, and pa20,
respectively.

For part (2), let $\mathcal M$ be a model of
$\mathrm{PA}_{\mathrm{smu}}^{-}$. The required differences, halves,
quotients, and cofactors corresponding to fixed nonzero divisors are unique
by the cancellation consequences of pa1--pa13, and the pa19 witness is
unique by pa20 and its defining inequalities. Indeed,
if $u\le v$ are two pa19 witnesses for the same $x$ and $u\ne v$, then
pa20 and discreteness give $2\cdot u\le v$, contradicting
$v\le x+1<2\cdot u$. Interpret
$\operatorname{sub}$ as cut-off subtraction, and use the corresponding
unique values for $\operatorname{half}$, $\operatorname{quo}$, and
$\operatorname{po}$ where applicable and $0$ otherwise. If
$\operatorname{pow}_2(x)$, set $\operatorname{obst}(x)=1$
and $\operatorname{cof}(x)=x$. Otherwise, choose an element $d_x$ such
that $d_x\mid x$, $d_x\ne1$, and $2\nmid d_x$, and set
$\operatorname{obst}(x)=d_x$. Let $\operatorname{cof}(x)$ be the unique
$c_x$ satisfying $c_x\cdot d_x=x$. Such a divisor $d_x$ exists by the
negation of $\operatorname{pow}_2(x)$. Thus Choice is used only to select
the obstruction divisors $d_x$. The resulting total functions give an
expansion of $\mathcal M$ satisfying $\mathrm{PA}_{\mathrm{wit}}^{-}$.
\end{proof}

\subsubsection{Substructures}

All axioms of $\mathrm{PA}_{\mathrm{wit}}^{-}$ are universal, and all the
new symbols are function symbols. Therefore any substructure closed under
all $\mathcal L_{\mathrm{wit}}$-functions is again a model of
$\mathrm{PA}_{\mathrm{wit}}^{-}$. Within that
substructure we still have
\[
\operatorname{obst}(x)=1\leftrightarrow\operatorname{pow}_2(x),
\]
so its $\mathcal L_0$-reduct is a model of
$\mathrm{PA}_{\mathrm{smu}}^{-}$.

\section{The standard model and the power-of-two sequence}

Let the sublanguage obtained by omitting $\operatorname{obst}$ and
$\operatorname{cof}$ be
\[
\mathcal L_{\mathrm{wit}}^{(0)}=
\mathcal L_0\cup
\{\operatorname{sub},
  \operatorname{po},
  \operatorname{quo},
  \operatorname{half}\}.
\]

Let
\[
\mathfrak N_0=(\mathbb N;0,1,+,\cdot,\le)
\]
be the $\mathcal L_0$-structure on the standard natural numbers, with all
symbols interpreted in the usual way.

We now define an $\mathcal L_{\mathrm{wit}}^{(0)}$-expansion
\[
\mathfrak N_{\mathrm{wit}}^{(0)}
\]
of $\mathfrak N_0$. First interpret $\operatorname{po}$ and
$\operatorname{quo}$ by
\[
\operatorname{po}^{\mathfrak N_{\mathrm{wit}}^{(0)}}(x)=
\max\{2^n\mid 2^n\le x+1\},
\]
\[
\operatorname{quo}^{\mathfrak N_{\mathrm{wit}}^{(0)}}(s,r)=
\begin{cases}
r/s&\text{if }s=2^m, r=2^n\text{ for some }m,n\in\mathbb N
                   \text{ and }s\le r,\\
0&\text{otherwise},
\end{cases}
\]
Interpret $\operatorname{sub}^{\mathfrak N_{\mathrm{wit}}^{(0)}}$ as
cut-off subtraction: its value at $(x,y)$ is $x-y$ if $y\le x$, and $0$
otherwise. Interpret
$\operatorname{half}^{\mathfrak N_{\mathrm{wit}}^{(0)}}(y)$ as $y/2$ if
$2\mid y$, and as $0$ otherwise. The quotients appearing in the definitions
of $\operatorname{quo}$ and $\operatorname{half}$ are standard natural
numbers whenever they are used.

Unless otherwise stated, terms and formulas in this section are evaluated
in $\mathfrak N_{\mathrm{wit}}^{(0)}$.

Let the ring of dyadic rationals be
\[
\mathbb D=\mathbb Z[1/2]
=\left\{\frac{a}{2^n}\;\middle|\;a\in\mathbb Z,
n\in\mathbb N\right\}.
\]

\subsection{The canonical expansion by an obstruction divisor and a power-of-two part}
Every nonzero $a\in\mathbb D$ can be written uniquely as
\[
a=2^e m,
\]
with $e\in\mathbb Z$ and $m\in\mathbb Z$ odd. Write $\nu_2(a)=e$. For
positive natural numbers this agrees with the usual $2$-adic valuation.

Define total functions $\operatorname{obst}_*$ and $\operatorname{cof}_*$
on the standard natural numbers by
\[
\operatorname{obst}_*(0)=3,
\qquad
\operatorname{cof}_*(0)=0,
\]
and, for $n\ge1$, by
\[
\operatorname{obst}_*(n)
=\frac{n}{2^{\nu_2(n)}},
\qquad
\operatorname{cof}_*(n)=2^{\nu_2(n)}.
\]
Thus $\operatorname{obst}_*(n)$ is the odd part of $n$, and
$\operatorname{cof}_*(n)$ is its power-of-two part.
Note that this canonical definition relies on the standard factorization
of natural numbers and is not assumed for arbitrary models of
$\mathrm{PA}_{\mathrm{smu}}^{-}$.
In particular,
\[
\operatorname{obst}_*(n)=1
\quad\Longleftrightarrow\quad
n=2^m\text{ for some }m\in\mathbb N.
\]

\begin{lem}[Canonical expansion by an obstruction divisor and a power-of-two part]\label{lem:canonical-obst-cof-extension}
Let $\mathfrak N_{\mathrm{wit}}^*$ be the expansion of
$\mathfrak N_{\mathrm{wit}}^{(0)}$ obtained by setting
\[
\operatorname{obst}=\operatorname{obst}_*,
\qquad
\operatorname{cof}=\operatorname{cof}_*.
\]
Then $\mathfrak N_{\mathrm{wit}}^*$ is a model of
$\mathrm{PA}_{\mathrm{wit}}^{-}$.
\end{lem}

\begin{proof}
pa1--pa13 hold on the standard natural numbers, while
$K_{\mathrm{sub}}$, $K_{\mathrm{fac}}$, and $K_{\mathrm{obst}}$ follow
immediately from the definitions.

Suppose $\operatorname{obst}(x)=1$, $z\cdot y=x$, and $y\ne1$.
Then $x$ is a power of two, so $2\mid y$ and $y/2\in\mathbb N$. Hence
$\operatorname{half}(y)=y/2$, proving $K_{\mathrm{half}}$.
The other two clauses of $K_{\mathrm{half}}$ follow directly from the
definition of $\operatorname{half}$ on even and odd inputs.

Let $u=\operatorname{po}(x)$. Since $u$ is the largest power of two not
exceeding $x+1$, we have $\operatorname{obst}(u)=1$ and $u\le x+1$.
If $2\cdot u\le x+1$, then
$2\cdot u$ would be a larger power of two not exceeding $x+1$, a
contradiction. Thus $x+1<2\cdot u$, proving $K_{\mathrm{po}}$. If
$\operatorname{obst}(s)=\operatorname{obst}(r)=1$ and $s\le r$,
then $s,r$ are powers of two and $r/s\in\mathbb N$. Hence
$\operatorname{quo}(s,r)=r/s$, proving the first clause of
$K_{\mathrm{quo}}$. Its second clause follows directly from the default
value $0$ in the definition of $\operatorname{quo}$.
\end{proof}

\subsection{Valuation and the power-of-two test along the power-of-two sequence}
Put
\[
\mathrm U_q=2^q
\qquad(q\in\mathbb N),
\]
and, for $q\ge2$, put
\[
\mathrm c_q=\frac54\mathrm U_q=5\cdot2^{q-2}.
\]
We call $(\mathrm U_q)_{q\in\mathbb N}$ the power-of-two sequence.

\begin{lem}[The lowest-degree term along the power-of-two sequence]\label{lem:power-sequence-valuation}
Let
\[
Q(T)=\sum_{r=h}^{d}a_{r}T^{r}\in\mathbb D[T],
\qquad
a_h\ne0
\]
be nonzero, where $h$ is the least degree with nonzero coefficient, and
put $e=\nu_2(a_h)$. Then, for all sufficiently large $q$, if $Q(\mathrm U_q)$ is
an integer, then
\[
\nu_2\bigl(Q(\mathrm U_q)\bigr)=e+qh.
\]
\end{lem}

\begin{proof}
Since $Q$ is nonzero, $Q(\mathrm U_q)\ne0$ for all sufficiently large $q$. Write
$a_h=2^e m$ with $m$ odd. Then
\[
\frac{Q(\mathrm U_q)}{2^{e}\mathrm U_q^{h}}
=m+
\sum_{r=h+1}^{d}2^{-e}a_{r}\mathrm U_q^{r-h}.
\]
Because every index in the sum satisfies $r>h$, each summand in
$\sum_{r=h+1}^{d}2^{-e}a_{r}\mathrm U_q^{r-h}$ is an even integer once $q$ is
sufficiently large. Since $m$ is odd, the entire right-hand side is odd,
and the conclusion follows.
\end{proof}

\begin{cor}[Obstruction divisors and power-of-two parts along the power-of-two sequence]\label{cor:power-sequence-obst-cof}
Let $Q(T)\in\mathbb D[T]$, and suppose $Q(\mathrm U_q)\in\mathbb N$ for all
sufficiently large $q$.

If $Q=0$, then
\[
\operatorname{obst}_*(Q(\mathrm U_q))=3,
\qquad
\operatorname{cof}_*(Q(\mathrm U_q))=0.
\]
If $Q\ne0$, let $h$ be its lowest degree, let $a_h$ be the corresponding
coefficient, and put $e=\nu_2(a_h)$. Then, for all sufficiently large $q$,
\[
\operatorname{obst}_*(Q(\mathrm U_q))
=2^{-e}\mathrm U_q^{-h}Q(\mathrm U_q),
\qquad
\operatorname{cof}_*(Q(\mathrm U_q))
=2^{e}\mathrm U_q^{h}.
\]
In particular, both right-hand sides are values at $\mathrm U_q$ of polynomials in
$\mathbb D[T]$.
\end{cor}

\begin{proof}
If $Q\ne0$, then $Q(\mathrm U_q)>0$ for all sufficiently large $q$. The formulas
follow from Lemma~\ref{lem:power-sequence-valuation}. Since $h$ is the lowest
degree, $2^{-e}T^{-h}Q(T)$ belongs to $\mathbb D[T]$.
\end{proof}

\begin{cor}[Stability of the power-of-two test along the power-of-two sequence]\label{cor:power-sequence-power-dichotomy}
Let $Q(T)\in\mathbb D[T]$, and suppose $Q(\mathrm U_q)\in\mathbb N$ for all
sufficiently large $q$. Then the truth value of
$\operatorname{obst}_*(Q(\mathrm U_q))=1$ is constant for all sufficiently large $q$.

Moreover, if this equation holds for all sufficiently large $q$, then for
some $e\in\mathbb Z$ and $h\in\mathbb N$,
\[
Q(T)=2^{e}T^{h}.
\]
If $h=0$, then $e\ge0$.
\end{cor}

\begin{proof}
If $Q=0$, then $\operatorname{obst}_*(Q(\mathrm U_q))=3\ne1$. Suppose $Q\ne0$ and, using the
notation of Corollary~\ref{cor:power-sequence-obst-cof}, put
\[
O_Q(T)=2^{-e}T^{-h}Q(T)\in\mathbb D[T].
\]
For sufficiently large $q$, $O_Q(\mathrm U_q)$ is a positive odd integer. Hence
$Q(\mathrm U_q)$ is a power of two exactly when $O_Q(\mathrm U_q)=1$. If $O_Q-1$ is the
zero polynomial, then $Q(T)=2^{e}T^{h}$; if $h=0$, the assumption that
$Q(\mathrm U_q)\in\mathbb N$ gives $2^e\in\mathbb N$, and hence $e\ge0$.
Otherwise $O_Q-1$ has only finitely many roots, so
$O_Q(\mathrm U_q)\ne1$ for all sufficiently large $q$.
\end{proof}

\begin{lem}[Stability of $\operatorname{po}$ along the power-of-two sequence]\label{lem:power-sequence-po}
Let $Q(T)\in\mathbb D[T]$, and suppose $Q(\mathrm U_q)\in\mathbb N$ for all
sufficiently large $q$. Then there exist $a\in\mathbb Z$ and
$h\in\mathbb N$ such that, for all sufficiently large $q$,
\[
\operatorname{po}^{\mathfrak N_{\mathrm{wit}}^{(0)}}(Q(\mathrm U_q))
=2^{a}\mathrm U_q^{h}.
\]
\end{lem}

\begin{proof}
Put $A(T)=Q(T)+1$. If $A$ is constant, then $Q(\mathrm U_q)$ is a constant natural
number and $\operatorname{po}(Q(\mathrm U_q))$ is a constant power of two. Thus,
for some integer $a\ge0$,
\[
\operatorname{po}(Q(\mathrm U_q))=2^{a}=2^{a}\mathrm U_q^{0}.
\]

Suppose $d=\deg A\ge1$. Since $A(\mathrm U_q)=Q(\mathrm U_q)+1>0$ for sufficiently large
$q$, the leading coefficient $b$ of $A$ is positive. Write
\[
A(T)=bT^d+R(T),
\qquad
\deg R<d.
\]
Then
\[
\frac{R(\mathrm U_q)}{\mathrm U_q^d}\longrightarrow0
\qquad(q\longrightarrow\infty).
\]
Choose an integer $a$ such that
\[
2^{a}\le b<2^{a+1}.
\]
After increasing $q$ if necessary, assume
\[
a-1+qd\ge0.
\]
This makes the exponent in $2^{a-1}\mathrm U_q^{d}=2^{a-1+qd}$ nonnegative. It also
gives $a+qd\ge0$, so both $2^{a}\mathrm U_q^{d}$ and $2^{a-1}\mathrm U_q^{d}$ below are
powers of two among the standard natural numbers.

First suppose $2^{a}<b<2^{a+1}$. Put
\[
\delta=\min\{b-2^{a},\,2^{a+1}-b\}>0.
\]
For all sufficiently large $q$,
\[
\left|\frac{R(\mathrm U_q)}{\mathrm U_q^d}\right|<\delta,
\]
and hence
\[
2^{a}\mathrm U_q^{d}<A(\mathrm U_q)<2^{a+1}\mathrm U_q^{d}.
\]
Therefore $\operatorname{po}(Q(\mathrm U_q))=2^{a}\mathrm U_q^{d}$.

Now suppose $b=2^{a}$. Then
\[
A(T)=2^{a}T^{d}+R(T),
\qquad
\deg R<d.
\]
If $R=0$, then $A(\mathrm U_q)=2^{a}\mathrm U_q^{d}$ and
$\operatorname{po}(Q(\mathrm U_q))=2^{a}\mathrm U_q^{d}$. If $R\ne0$, the sign of
$R(\mathrm U_q)$ is constant for all sufficiently large $q$.

If $R(\mathrm U_q)>0$ eventually, then $R(\mathrm U_q)/\mathrm U_q^d\to0$, so for sufficiently
large $q$,
\[
0<\frac{R(\mathrm U_q)}{\mathrm U_q^d}<2^{a},
\]
and hence
\[
2^{a}\mathrm U_q^{d}<A(\mathrm U_q)<2^{a+1}\mathrm U_q^{d}.
\]
Thus $\operatorname{po}(Q(\mathrm U_q))=2^{a}\mathrm U_q^{d}$.

If $R(\mathrm U_q)<0$ eventually, then for sufficiently large $q$,
\[
\left|\frac{R(\mathrm U_q)}{\mathrm U_q^d}\right|<2^{a-1}.
\]
Therefore
\[
2^{a-1}\mathrm U_q^{d}<A(\mathrm U_q)<2^{a}\mathrm U_q^{d},
\]
and $\operatorname{po}(Q(\mathrm U_q))=2^{a-1}\mathrm U_q^{d}$.
\end{proof}

\subsection{Polynomial representation of $\mathcal L_{\mathrm{wit}}$-terms along the power-of-two sequence}
\begin{lem}[Polynomial representation of terms along the power-of-two sequence]\label{lem:full-language-power-sequence-polynomial}
For every $\mathcal L_{\mathrm{wit}}$-term $t(x)$ that is either closed or
has $x$ as its only free variable, there exist a polynomial
\[
\mathrm Q_t(T)\in\mathbb D[T]
\]
and a natural number $q_t$ such that, for every $q\ge q_t$,
\[
t^{\mathfrak N_{\mathrm{wit}}^*}(\mathrm c_q)=\mathrm Q_t(\mathrm U_q).
\]
\end{lem}

\begin{proof}
We argue by induction on the construction of terms. When treating a
compound term, apply the induction hypothesis to all its proper subterms.
Thus we may take $q$ sufficiently large that, simultaneously for every
proper subterm $u$,
\[
u^{\mathfrak N_{\mathrm{wit}}^*}(\mathrm c_q)=\mathrm Q_u(\mathrm U_q)
\]
holds.

For the terms $0,1,x$, take
\[
\mathrm Q_0(T)=0,
\qquad
\mathrm Q_1(T)=1,
\qquad
\mathrm Q_x(T)=\frac54T.
\]
For addition and multiplication, take the sum and product of the
corresponding polynomials.

Suppose $t=\operatorname{sub}(u,v)$. By the induction hypothesis, for all
sufficiently large $q$,
\[
u(\mathrm c_q)-v(\mathrm c_q)=(\mathrm Q_u-\mathrm Q_v)(\mathrm U_q).
\]
Both sides are integers and have the same sign, so
\[
v(\mathrm c_q)\le u(\mathrm c_q)
\quad\Longleftrightarrow\quad
(\mathrm Q_u-\mathrm Q_v)(\mathrm U_q)\ge0.
\]
Either $\mathrm Q_u-\mathrm Q_v$ is the zero polynomial, or its value at $\mathrm U_q$ has constant
sign for all sufficiently large $q$. Thus the truth value of this condition
is eventually constant. If it is true, take
\[
\mathrm Q_t=\mathrm Q_u-\mathrm Q_v,
\]
and if it is false, take $\mathrm Q_t=0$.

If $t=\operatorname{po}(u)$, the induction hypothesis gives
$\mathrm Q_u(\mathrm U_q)=u(\mathrm c_q)\in\mathbb N$ for all sufficiently large $q$. The required
polynomial is supplied by Lemma~\ref{lem:power-sequence-po} applied to $\mathrm Q_u$.

Suppose $t=\operatorname{quo}(s,r)$. By the induction hypothesis and
Corollary~\ref{cor:power-sequence-power-dichotomy}, the truth values of the two
equations
\[
\operatorname{obst}(s(\mathrm c_q))=1,
\qquad
\operatorname{obst}(r(\mathrm c_q))=1
\]
are eventually constant. Moreover, for all
sufficiently large $q$,
\[
s(\mathrm c_q)\le r(\mathrm c_q)
\quad\Longleftrightarrow\quad
(\mathrm Q_r-\mathrm Q_s)(\mathrm U_q)\ge0,
\]
so this truth value is eventually constant as well. If any condition in
the definition of $\operatorname{quo}$ is eventually false, take $\mathrm Q_t=0$.
If they are all eventually true, the same corollary gives
\[
\mathrm Q_s(T)=2^{a_s}T^{h_s},
\qquad
\mathrm Q_r(T)=2^{a_r}T^{h_r}.
\]
Since $s(\mathrm c_q)\le r(\mathrm c_q)$ eventually, $h_r\ge h_s$. Set
\[
\mathrm Q_t(T)=2^{a_r-a_s}T^{h_r-h_s}.
\]
Then, for all sufficiently large $q$,
\[
\mathrm Q_t(\mathrm U_q)\,\mathrm Q_s(\mathrm U_q)=\mathrm Q_r(\mathrm U_q),
\]
and hence $t(\mathrm c_q)=\mathrm Q_t(\mathrm U_q)$.

Suppose $t=\operatorname{half}(u)$. By the induction hypothesis,
$u(\mathrm c_q)=\mathrm Q_u(\mathrm U_q)\in\mathbb N$ for all sufficiently large $q$.
If $\mathrm Q_u=0$, take $\mathrm Q_t=0$. If $\mathrm Q_u\ne0$,
Lemma~\ref{lem:power-sequence-valuation} shows that the parity of $\mathrm Q_u(\mathrm U_q)$
is eventually constant. If it is eventually odd, take $\mathrm Q_t=0$. If it is
eventually even, then $\operatorname{half}(u(\mathrm c_q))=u(\mathrm c_q)/2$, so take
\[
\mathrm Q_t(T)=\frac12\mathrm Q_u(T)\in\mathbb D[T].
\]

If $t=\operatorname{obst}(u)$ or $t=\operatorname{cof}(u)$, the induction
hypothesis gives $\mathrm Q_u(\mathrm U_q)=u(\mathrm c_q)\in\mathbb N$ for all sufficiently large
$q$. The result follows by applying
Corollary~\ref{cor:power-sequence-obst-cof} to $\mathrm Q_u$.

Every term has only finitely many subterms. Taking $q_t$ to be the maximum
of the finitely many bounds appearing above makes all conditions hold
simultaneously.
\end{proof}

\begin{lem}[Finite avoidance for $\mathcal L_{\mathrm{wit}}$-terms]\label{lem:full-language-finite-avoidance}
Let
\[
t_1(x),\ldots,t_m(x)
\]
be finitely many $\mathcal L_{\mathrm{wit}}$-terms that are either closed
or have $x$ as their only free variable. Then there is a standard natural
number $c$ such that, in $\mathfrak N_{\mathrm{wit}}^*$, for every
$1\le i\le m$ and $j\in\{0,1,2\}$,
\[
c\ne3t_i(c)+j.
\]
\end{lem}

\begin{proof}
By Lemma~\ref{lem:full-language-power-sequence-polynomial}, choose
$\mathrm Q_{t_i}(T)\in\mathbb D[T]$ for each $i$. For each $i$ and
$j\in\{0,1,2\}$, put
\[
R_{i,j}(T)
=\frac54T-3\mathrm Q_{t_i}(T)-j
\in\mathbb D[T].
\]

$R_{i,j}$ is nonzero. Indeed, if $R_{i,j}=0$, then
\[
\mathrm Q_{t_i}(T)=\frac{5}{12}T-\frac j3.
\]
But $5/12\notin\mathbb D$, contradicting
$\mathrm Q_{t_i}(T)\in\mathbb D[T]$.

Each of the finitely many nonzero polynomials $R_{i,j}$ has only finitely
many roots. We may therefore choose $q$ sufficiently large that
$t_i(\mathrm c_q)=\mathrm Q_{t_i}(\mathrm U_q)$ for every $i$, and also
\[
R_{i,j}(\mathrm U_q)\ne0,
\]
for every $i,j$. For $c=\mathrm c_q$ we then have
\[
c-3t_i(c)-j
=R_{i,j}(\mathrm U_q)\ne0.
\]
\end{proof}

\section{Compactness and the generated substructure}

We combine the finite avoidance lemma with the Compactness Theorem to
obtain an element satisfying the avoidance conditions for all terms
simultaneously. This yields a model containing an element for which no
quotient and remainder exist upon division by $3$.

Let $\operatorname{Th}_{\forall}(\mathfrak N_0)$ denote the set of all
universal $\mathcal L_0$-sentences true in the standard model
$\mathfrak N_0$. The construction below preserves this entire universal
theory.

Let $\mathcal L_{\mathrm{wit}}(\mathsf c)$ be the language obtained from
$\mathcal L_{\mathrm{wit}}$ by adding a new constant symbol $\mathsf c$.
Let $\mathcal T_x$ be the set of all $\mathcal L_{\mathrm{wit}}$-terms that
are either closed or have $x$ as their only free variable. For
$t\in\mathcal T_x$, let $t(\mathsf c)$ denote the closed term obtained by
replacing $x$ with $\mathsf c$. Consider the following
$\mathcal L_{\mathrm{wit}}(\mathsf c)$-theory:
\[
\begin{split}
T_{\mathrm{avoid}}
=\;&\mathrm{PA}_{\mathrm{wit}}^{-}\\
&\cup\operatorname{Th}_{\forall}(\mathfrak N_0)\\
&\cup
\{\mathsf c\ne3t(\mathsf c)+j\mid
  t\in\mathcal T_x, j\in\{0,1,2\}\}.
\end{split}
\]

\begin{thm}[A model in which division by $3$ fails]\label{thm:mod-three-gap}
There exist a model $\mathcal A$ of
$\mathrm{PA}_{\mathrm{smu}}^{-}\cup
\operatorname{Th}_{\forall}(\mathfrak N_0)$
and an element $d\in A$ such that, for every $a\in A$ and
$j\in\{0,1,2\}$,
\[
d\ne3a+j.
\]
\end{thm}

\begin{proof}
Let $T_0$ be any finite subset of $T_{\mathrm{avoid}}$, and list the
finitely many terms appearing in its avoidance conditions as
\[
t_1(x),\ldots,t_m(x).
\]

By Lemma~\ref{lem:full-language-finite-avoidance}, there is a standard
natural number $c$ which, in $\mathfrak N_{\mathrm{wit}}^*$, satisfies all
avoidance conditions occurring in $T_0$; that is, for every
$1\le i\le m$ and $j\in\{0,1,2\}$,
\[
c\ne3t_i(c)+j.
\]
Interpreting $\mathsf c$ as this $c$ makes
$\mathfrak N_{\mathrm{wit}}^*$ a model of $T_0$, since its
$\mathcal L_0$-reduct is $\mathfrak N_0$ and therefore satisfies every
sentence in $T_0\cap\operatorname{Th}_{\forall}(\mathfrak N_0)$. Hence
$T_{\mathrm{avoid}}$ is finitely satisfiable.

By the Compactness Theorem, $T_{\mathrm{avoid}}$ has a model
\[
(\mathcal M,d),
\]
where $d=\mathsf c^{\mathcal M}$.

Let $\mathcal H$ be the $\mathcal L_{\mathrm{wit}}$-substructure of
$\mathcal M$ generated by $d$. Its underlying set is
\[
H=\{t^{\mathcal M}(d)\mid t\in\mathcal T_x\}.
\]
Since $x\in\mathcal T_x$, we have $d\in H$. Moreover, all axioms of
$\mathrm{PA}_{\mathrm{wit}}^{-}$ and all sentences in
$\operatorname{Th}_{\forall}(\mathfrak N_0)$ are universal. Hence
$\mathcal H$ is also a model of
$\mathrm{PA}_{\mathrm{wit}}^{-}\cup
\operatorname{Th}_{\forall}(\mathfrak N_0)$.

Take $a\in H$. For some $\mathcal L_{\mathrm{wit}}$-term $t(x)$,
\[
a=t^{\mathcal M}(d).
\]
By the definition of $T_{\mathrm{avoid}}$, for every $j\in\{0,1,2\}$,
\[
d\ne3t^{\mathcal M}(d)+j.
\]
Since $a=t^{\mathcal M}(d)$,
\[
3t^{\mathcal M}(d)+j=3a+j,
\]
and therefore
\[
d\ne3a+j.
\]
These relations continue to hold in the substructure $\mathcal H$.

Finally, let $\mathcal A$ be the $\mathcal L_0$-reduct of $\mathcal H$.
Since $\mathcal H\models\mathrm{PA}_{\mathrm{wit}}^{-}$,
Lemma~\ref{lem:reduct} gives
$\mathcal A\models\mathrm{PA}_{\mathrm{smu}}^{-}$. Moreover,
$\mathcal H\models\operatorname{Th}_{\forall}(\mathfrak N_0)$, and these
sentences belong to $\mathcal L_0$, so they remain true in the reduct
$\mathcal A$. Hence
\[
\mathcal A\models
\mathrm{PA}_{\mathrm{smu}}^{-}
\cup\operatorname{Th}_{\forall}(\mathfrak N_0).
\]
Together with the avoidance property established above, this completes the
proof.
\end{proof}

\begin{cor}[Non-derivability of the Standard Euclidean Division Principle]\label{cor:standard-division-not-provable}
\[
\mathrm{PA}_{\mathrm{smu}}^{-}
\cup\operatorname{Th}_{\forall}(\mathfrak N_0)
\nvdash\mathrm{pa16}_3.
\]
Consequently, this stronger theory does not prove the scheme pa16; in
particular, neither does $\mathrm{PA}_{\mathrm{smu}}^{-}$.
\end{cor}

\begin{proof}
Take the model $\mathcal A$ and element $d$ from
Theorem~\ref{thm:mod-three-gap}. Let $q,r\in A$ and suppose that $r<3$.
Then $r$ is one of $0,1,2$. The theorem's avoidance condition and pa5
therefore give
\[
d\ne q\cdot3+r.
\]
Thus $\mathrm{pa16}_3$ is false in $\mathcal A$, and the Soundness Theorem
implies that it is not provable from
$\mathrm{PA}_{\mathrm{smu}}^{-}\cup
\operatorname{Th}_{\forall}(\mathfrak N_0)$.
Since even the $n=3$ instance is unprovable,
this stronger theory does not prove the scheme pa16.
\end{proof}

\begin{cor}[Non-derivability of the Euclidean Division Principle]\label{cor:division-not-provable}
\[
\mathrm{PA}_{\mathrm{smu}}^{-}
\cup\operatorname{Th}_{\forall}(\mathfrak N_0)
\nvdash\mathrm{pa17}.
\]
In particular, $\mathrm{PA}_{\mathrm{smu}}^{-}\nvdash\mathrm{pa17}$.
\end{cor}

\begin{proof}
Take the model $\mathcal A$ and element $d$ from
Theorem~\ref{thm:mod-three-gap}. $\mathcal A\models3\ne0$. If pa17 held in
$\mathcal A$, then with $x=d$
and $y=3$ there would be $q,r\in A$ such that
\[
r<3,
\qquad
d=q\cdot3+r.
\]
Since $r$ is one of $0,1,2$, this contradicts the avoidance condition in
Theorem~\ref{thm:mod-three-gap}. Thus pa17 is false
in $\mathcal A$, and the Soundness Theorem implies that it is not provable
from $\mathrm{PA}_{\mathrm{smu}}^{-}\cup
\operatorname{Th}_{\forall}(\mathfrak N_0)$.
\end{proof}

\section*{Acknowledgements}

The author is deeply grateful to Albert Visser for posing the question
addressed in this paper, for carefully reading an earlier version of the
manuscript, and for his valuable comments, suggestions, and encouragement.

\section*{Statement on the Use of AI Tools}

In preparing this paper, the author used OpenAI Codex, based on GPT-5, as an
interactive tool to assist in exploring proof strategies, checking
mathematical arguments, organizing and revising the exposition, and preparing
the English version. The author independently verified all definitions,
proofs, references, and conclusions and takes full responsibility for the
entire content of the paper.

\end{document}